\documentclass[12pt]{amsart}
\usepackage{amsfonts}

\newtheorem{theorem}{Theorem}[section]
\newtheorem{corollary}[theorem]{Corollary}

\newtheorem{lemma}[theorem]{Lemma}
\newtheorem{proposition}[theorem]{Proposition}

\newtheorem{remark}[theorem]{Remark}

\def\NN{\hbox{\sf I\kern-.13em\hbox{N}}}
\def\RR{\hbox{\sf I\kern-.14em\hbox{R}}}

\def\Cc{\hbox{\sf C\kern -.47em {\raise .48ex \hbox{$\scriptscriptstyle |$}}
   \kern-.5em {\raise .48ex \hbox{$\scriptscriptstyle |$}} }}

\newcommand{\be}{\begin{equation}}
\newcommand{\ee}{\end{equation}}

\newcommand{\cM}{{\mathcal M}}

\begin{document}

\title[Hadamard weighted geometric mean inequalities]{ Hadamard weighted geometric mean inequalities for the spectral and essential spectral radius of positive operators on Banach function and sequence spaces} 

\author{  Katarina Bogdanovi\'{c}$^{1}$, Aljo\v{s}a Peperko$^{2,3,*}$}
\date{\thanks{
Faculty of Mathematics$^1$,
University of Belgrade,
Studentski trg 16,
SRB-11000 Belgrade, Serbia,
email:   katarinabgd77@gmail.com \\ 
Faculty of Mechanical Engineering$^2$,
University of Ljubljana,
A\v{s}ker\v{c}eva 6,
SI-1000 Ljubljana, Slovenia,\\
Institute of Mathematics, Physics and Mechanics$^3$,
Jadranska 19,
SI-1000 Ljubljana, Slovenia \\
e-mail:   aljosa.peperko@fs.uni-lj.si \\
* Corresponding author
} \today}

\begin{abstract}
\baselineskip 7mm
We prove new inequalities for the spectral radius, essential spectral radius, operator norm, measure of noncompactness and numerical radius of Hadamard weighted geometric means of positive kernel operators on Banach function and sequence spaces.   Several inequalities appear to be new even in the finite dimensional case. 
\end{abstract}

\baselineskip 6,5mm

\maketitle

\noindent
{\it Math. Subj.  Classification (2010)}: {\it \small Primary:} 47A10;  47B65; 47A63; \\ 
\noindent {\it \small Secondary:} 46E30; 46B45; 15A42; 15A60.  

\section{Introduction and preliminaries}

\vspace{5mm}

Let $\mu$ be a $\sigma$-finite positive measure on a $\sigma$-algebra $\cM$ of subsets of a non-void set $X$.
Let $M(X,\mu)$ be the vector space of all equivalence classes of (almost everywhere equal)
complex measurable functions on $X$. A Banach space $L \subseteq M(X,\mu)$ is
called a {\it Banach function space} if $f \in L$, $g \in M(X,\mu)$,
and $|g| \le |f|$ imply that $g \in L$ and $\|g\| \le \|f\|$. We assume that  $X$ is the carrier of $L$, that is, there is no subset $Y$ of $X$ of 
 strictly positive measure with the property that $f = 0$ a.e. on $Y$ for all $f \in L$ (see \cite{Za83}).
 
 Let $R$ denote the set $\{1, \ldots, N\}$ for some $N \in \NN$ or the set $\NN$ of all natural numbers.
Let $S(R)$ be the vector lattice of all complex sequences $(x_n)_{n\in R}$.
A Banach space $L \subseteq S(R)$ is called a {\it Banach sequence space} if $x \in S(R)$, $y \in L$
and $|x| \le |y|$ imply that $x \in L$ and $\|x\|_L \le \|y\|_L$. Observe that a Banach sequence space is a Banach function space over a measure space $(R, \mu)$, 
where $\mu$ denotes the counting measure on $R$. Denote by $\mathcal{L}$ the collection of all Banach sequence spaces
$L$ satisfying the property that $e_n = \chi_{\{n\}} \in L$ and
$\|e_n\|_L=1$ for all $n \in R$. For $L\in \mathcal{L}$ the set $R$ is the carrier of $L$.

%
 

Standard examples of Banach function spaces are Euclidean spaces,  the space $c_0$ 
of all null convergent sequences  (equipped with the usual norms and the counting measure), the
well-known spaces $L^p (X,\mu)$ ($1\le p \le \infty$) and other less known examples such as Orlicz, Lorentz,  Marcinkiewicz  and more general  rearrangement-invariant spaces (see e.g. \cite{BS88}, \cite{CR07} and the references cited there), which are important e.g. in interpolation theory.
 Recall that the cartesian product $L=E\times F$ 
of Banach function spaces is again a Banach function space, equipped with the norm
$\|(f, g)\|_L=\max \{\|f\|_E, \|g\|_F\}$.

If $\{f_n\}_{n\in \mathbb{N}} \subset M(X,\mu)$ is a decreasing sequence and
$f=\inf\{f_n \in M(X,\mu): n \in \NN \}$, then we write $f_n \downarrow f$. 
A Banach function space $L$ has an {\it order continuous norm}, if $0\le f_n \downarrow 0$
implies $\|f_n\|_L \to 0$ as $n \to \infty$. It is well known that spaces $L^p$, $1\le p< \infty$, have order continuous
norm. Moreover, the norm of any reflexive Banach function space is
order continuous. 
In particular, we will be interested in  Banach function spaces $L$ such that $L$ and its Banach dual space $L^*$ have order continuous norms. Examples of such spaces are $L^p  (X,\mu)$, $1< p< \infty$, while the space
$L=c_0\in \mathcal{L}$ 
is an example of a non-reflexive Banach sequence (function) space, such that $L$ and  $L^*=l^1\in \mathcal{L}$ have order continuous
norms.

By an {\it operator} on a Banach function space $L$ we always mean a linear
operator on $L$.  An operator $K$ on $L$ is said to be {\it positive} 
if it maps nonnegative functions to nonnegative ones, i.e., $KL_+ \subset L_+$, where $L_+$ denotes the positive cone $L_+ =\{f\in L : f\ge 0 \; \mathrm{a.e.}\}$.
Given operators $K$ and $H$ on $L$, we write $K \ge H$ if the operator $K - H$ is positive.
Recall that a positive  operator $K$ is always bounded, i.e., its operator norm
\be
\|K\|=\sup\{\|Kf\|_L : f\in L, \|f\|_L \le 1\}=\sup\{\|Kf\|_L : f\in L_+, \|f\|_L \le 1\}
\label{equiv_op}
\ee
is finite (the second equality in (\ref{equiv_op}) follows from $|Kf| \le K|f|$ for $f\in L$).  
Also, its spectral radius $r(K)$ is always contained in the spectrum.

In the special case $L= L^2(X, \mu)$ we can define the {\it numerical radius} $w(K)$ of 
a bounded operator $K$ on $L^2(X, \mu)$ by 
$$ w(K) = \sup \{ | \langle K f, f \rangle | : f \in L^2(X, \mu), \| f \|_2 = 1 \} . $$
If, in addition, $K$ is positive, then it is easy to prove that 
$$ w(K) = \sup \{ \langle K f, f \rangle  : f \in L^2(X, \mu)_+ , \| f \|_2 = 1 \} . $$
From this it follows easily that $w(K) \le w(H)$ for all positive operators $K$ and $H$ on $L^2(X, \mu)$ with $K \le H$.


An operator $K$ on a Banach function space $L$ is called a {\it kernel operator} if
there exists a $\mu \times \mu$-measurable function
$k(x,y)$ on $X \times X$ such that, for all $f \in L$ and for almost all $x \in X$,
$$ \int_X |k(x,y) f(y)| \, d\mu(y) < \infty \ \ \ {\rm and} \ \ 
   (Kf)(x) = \int_X k(x,y) f(y) \, d\mu(y)  .$$
One can check that a kernel operator $K$ is positive iff 
its kernel $k$ is non-negative almost everywhere. 

Let $L$ be a Banach function space such that $L$ and $L^*$ have order
continuous norms and let $K$ and $H$ be  positive kernel operators on $L$. By $\gamma (K)$ we denote the Hausdorff measure of 
non-compactness of $K$, i.e., 
$$\gamma (K) = \inf\left\{ \delta >0 : \;\; \mathrm{there}\;\; \mathrm{is} \;\; \mathrm{a}\;\; \mathrm{finite}\;\; M \subset L \;\;\mathrm{such} \;\; \mathrm{that} \;\; K(D_L) \subset M + \delta D_L  \right\},$$
where $D_L =\{f\in L : \|f\|_L \le 1\}$. Then $\gamma (K) \le \|K\|$, $\gamma (K+H) \le \gamma (K) + \gamma (H)$, $\gamma(KH) \le \gamma (K)\gamma (H)$ and $\gamma (\alpha K) =\alpha \gamma (K)$ for $\alpha \ge 0$. Also 
$0 \le K\le H$  implies $\gamma (K) \le \gamma (H)$ (see e.g. \cite[Corollary 4.3.7 and Corollary 3.7.3]{Me91}). Let $r_{ess} (K)$ denote the essential spectral radius of $K$, i.e., the spectral radius of the Calkin image of $K$ in the Calkin algebra. Then 
\be
 r_{ess} (K) =\lim _{j \to \infty} \gamma (K^j)^{1/j}=\inf _{j \in \NN} \gamma (K^j)^{1/j} 
\label{esslim=inf}
\ee
and $r_{ess} (K) \le \gamma (K)$. Note that (\ref{esslim=inf}) is valid for any bounded operator $K$ on a given complex Banach space $L$ (see e.g. \cite[Theorem 4.3.13]{Me91}).


It is well-known that kernel operators play a very important, often even central, role in a variety of applications from differential and integro-differential equations, problems from physics 
(in particular from thermodinamics), engineering, statistical and economic models, etc (see e.g. \cite{J82}, \cite{BP03}, \cite{LL05}, \cite{DLR13} 
and the references cited there).
For the theory of Banach function spaces and more general Banach lattices we refer the reader to the books \cite{Za83}, \cite{BS88}, \cite{AA02}, \cite{AB85}. 

Let $K$ and $H$ be positive kernel operators on $L$ with kernels $k$ and $h$ respectively,
and $\alpha \ge 0$.
The \textit{Hadamard (or Schur) product} $K \circ H$ of $K$ and $H$ is the kernel operator
with kernel equal to $k(x,y)h(x,y)$ at point $(x,y) \in X \times X$ which can be defined (in general) 
only on some order ideal of $L$. Similarly, the \textit{Hadamard (or Schur) power} 
$K^{(\alpha)}$ of $K$ is the kernel operator with kernel equal to $(k(x, y))^{\alpha}$ 
at point $(x,y) \in X \times X$ which can be defined only on some order ideal of $L$.

Let $K_1 ,\ldots, K_n$ be positive kernel operators on a Banach function space $L$, 
and $\alpha _1, \ldots, \alpha _n$ positive numbers such that $\sum_{j=1}^n \alpha _j = 1$.
Then the {\it  Hadamard weighted geometric mean} 
$K = K_1 ^{( \alpha _1)} \circ K_2 ^{(\alpha _2)} \circ \cdots \circ K_n ^{(\alpha _n)}$ of 
the operators $K_1 ,\ldots, K_n$ is a positive kernel operator defined 
on the whole space $L$, since $K \le \alpha _1 K_1 + \alpha _2 K_2 + \ldots + \alpha _n K_n$ by the inequality between the weighted arithmetic and geometric means.

A matrix $K=[k_{ij}]_{i,j\in R}$ is called {\it nonnegative} if $k_{ij}\ge 0$ for all $i, j \in R$. 
For notational convenience, we sometimes write $k(i,j)$ instead of $k_{ij}$. 
We say that a nonnegative matrix $K$ defines an operator on $L$ if $Kx \in L$ for all $x\in L$, where
$(Kx)_i = \sum _{j \in R}k_{ij}x_j$. Then $Kx \in L_+$ for all $x\in L_+$ and
so $K$ defines a {\rm positive} kernel operator on $L$.

Let us recall  the following result which was proved in \cite[Theorem 2.2]{DP05} and 
\cite[Theorem 5.1 and Example 3.7]{P06} (see also e.g. \cite[Theorem 2.1]{P17}, \cite{P18a}, 
\cite{P11}, \cite{P12}, \cite{DP16}).

\begin{theorem}
\label{good_work}
Let $K, K_1, \ldots , K_n$ and 
$\{K_{i j}\}_{i=1, j=1}^{l, n}$ 
be positive kernel operators on a Banach function space $L$.
Assume $\alpha _1$, $\alpha _2$,..., $\alpha _n$ are positive numbers  
such that $s_n=\sum_{j=1}^n \alpha _j = 1$ and define 
the  positive kernel operator $H$ on $L$ by
$$H:= \left(K_{1 1}^{(\alpha _1)} \circ \cdots \circ K_{1 n}^{(\alpha _n)}\right) \ldots \left(K_{l 1}^{(\alpha _1)} \circ \cdots \circ K_{l n}^{(\alpha _n)} \right).$$

\noindent (i) Then the following inequalities hold for $\rho \in \{\|\cdot\|, r\}$:

\be
 \rho(K_1 ^{( \alpha _1)} \circ K_2 ^{(\alpha _2)} \circ \cdots \circ K_n ^{(\alpha _n)} ) \le
\rho(K_1)^{ \alpha _1} \, \rho(K_2)^{\alpha _2} \cdots \rho(K_n)^{\alpha _n} ,
\label{gl1vecr}
\ee
\be
\label{basic2}
H \le  
(K_{1 1} \cdots  K_{l 1})^{(\alpha _1)} \circ \cdots 
\circ (K_{1 n} \cdots K_{l n})^{(\alpha _n)} , \\
\ee

\begin{eqnarray}
\rho \left(H \right)   &\le &  
\nonumber
\rho \left( (K_{1 1} \cdots  K_{l 1})^{(\alpha _1)} \circ \cdots 
\circ (K_{1 n} \cdots K_{l n})^{(\alpha _n)} \right) \\
&\le  &
\rho \left( K_{1 1} \cdots  K_{l 1} \right)^{\alpha _1} \cdots 
\rho \left( K_{1 n} \cdots K_{l n}\right)^{\alpha _n} .
\label{very_good}
\end{eqnarray}

If, in addition, $L$ and $L^*$ have order continuous norms, 
then  (\ref{gl1vecr}) and (\ref{very_good}) hold also for $\rho \in \{\gamma, r_{ess}\}$.

If, in addition, $L=L^2(X, \mu)$ 
then   (\ref{gl1vecr}) and (\ref{very_good}) hold also for $\rho =w$.

\noindent (ii) If $L \in \mathcal{L}$, 
$t \ge 1$ and $s_n \ge 1$, then   
 $K^{(t)}$, $K_1 ^{( \alpha _1)}
\circ K_2 ^{(\alpha _2)} \circ \cdots \circ K_n ^{(\alpha _n)}$ and $H$ define  operators on $L$ and the inequalities  
\be
k(i,j)\le \|K\| \;\;\mathit{for} \;\;\mathit{all}\;\; i,j\in R,
\label{lg} 
\ee 
\be
 K_{1} ^{(t)} \cdots  K_{n} ^{(t)} \le ( K_1 \cdots K_n )^{(t)}, 
\label{t_dobro}
\ee
 \be
 \rho (K_{1} ^{(t)} \cdots  K_{n} ^{(t)}) \le \rho (K_1 \cdots K_n)^{t}, 
 \label{gl3tr}
 \ee
(\ref{gl1vecr}) and (\ref{very_good}) hold for $\rho \in \{\|\cdot\|, r\}$.
\end{theorem}

In the finite-dimensional case  Inequality (\ref{gl1vecr}) for the spectral radius goes back to Kingman \cite{K61} implicitly,
and it was later considered by several authors (\cite{HJ91}, \cite{BR97}, \cite{EJS88}, \cite{KO85}, \cite{Sch86},
\cite{D92}, \cite{HLS97}, \cite{H07}),  using different methods. 
In \cite{EJS88}  (see also \cite{HJ91}) the method of linearization was applied 
to generalize the result of Cohen  (see \cite{C79}, \cite{C81}, \cite{F81}, \cite{E84}, \cite{DN84}, \cite[Corollary 5.7.13]{HJ91}, \cite{EJS88}, \cite[Theorem 3.5.9]{BR97}) that asserts that the spectral radius, considered as a function of the diagonal entries of a nonnegative matrix, 
is a convex function. In \cite{DP10} the generalization to the setting of infinite dimensional matrices was established and very recently in \cite{P21a} a version for the essential spectral radius was proved.
It should be mentioned that a very general extension of Cohen's theorem was proved in the setting
of Banach ordered spaces by Kato \cite{Ka82}, who used extensively the theory of strongly continuous semigroups of operators.
Let us also point out that Inequalities 
(\ref{very_good}) are actually results on the joint and generalized spectral radius and their essential versions (see e.g.  \cite{P06}, \cite{P12}, \cite{P17}, \cite{P18b}) 
and have been applied to obtain several inequalities involving the Hadamard and ordinary products of operators (see, e.g., \cite{Sch11}, \cite{Sch11b}, \cite{P12}, \cite{CZ16}, \cite{DP16}, \cite{P17}, \cite{P18a}, \cite{P18b},  \cite{Z18}, \cite{MP13}, \cite{MP18}, \cite{RLP18}, \cite{P21a}).

\vspace{4mm}

Let $K_1=[k_1(i,j)]_{i,j\in R}, \ldots, K_n=[k_n(i,j)]_{i,j\in R}$
be nonnegative matrices and let $\alpha_1, \ldots, \alpha_n$ be nonnegative numbers such that $\sum_{i=1}^n \alpha_i =1$. The nonnegative matrix \\
$C(K_1, \ldots , K_n, \alpha_1, \ldots , \alpha_n) = [c(i,j)]_{i,j\in R}$ (\cite{DP10}, \cite{EJS88}, \cite{HJ91}) is defined by 

$$ c(i,j) = \left\{ \begin{array}{ccc}
k_1 ^{\alpha_1} (i,j) \cdots k_n ^{\alpha_n} (i,j) & \textrm{if} & i \neq j \\
\alpha_1 k_1(i,i)+ \ldots + \alpha_n k_n(i,i) & \textrm{if} & i = j 
\end{array} \right. .$$
In other words, the diagonal part of $C(K_1, \ldots , K_n, \alpha_1, \ldots , \alpha_n)$ is 
equal to the diagonal part of $\alpha_1 K_1 + \cdots + \alpha_n K_n$, while 
its nondiagonal part equals the nondiagonal part of $K_1^{( \alpha_1)} \circ K_2^{(\alpha_2)} \circ \cdots \circ K_n^{(\alpha_n)}$.

By the inequality between weighted geometric and weighted arithmetic means, we have
\be
 K_1 ^{( \alpha _1)} \circ K_2 ^{(\alpha _2)} \circ \cdots \circ K_n ^{(\alpha _n)} \le 
   C(K_1,\ldots , K_n, \alpha_1, \ldots , \alpha _n) \le  \alpha_1 K_1 + \cdots + \alpha _n K_n . 
\label{c_ineq}
\ee
From the right-hand inequality it follows that the matrix $C(K_1,\ldots , K_n, \alpha_1, \ldots , \alpha
_n)$ defines an operator on $L$ provided 
the matrices $K_1,\ldots,K_n$ define operators on $L\in \mathcal{L}$.

The following generalization of Cohen's theorem was obtained in \cite[Theorem 2.1]{DP10} and in \cite[Theorem 2.2]{P21a}.

\begin{theorem}
Given $L$ in $\mathcal{L}$, let $K_1, \ldots , K_n$ be nonnegative
matrices that define operators on $L$ and $\alpha_1, \ldots, \alpha
_n$ nonnegative numbers such that $\sum_{i=1} ^n \alpha_i =1$. Then for $\rho=r$ we have 
\be 
\rho \left(C(K_1,\ldots , K_n, \alpha_1, \ldots , \alpha_n)\right) \le 
\alpha _1\rho (K_1)+ \cdots + \alpha _n \rho (K_n).
\label{ckh} 
\ee
In particular, if $K_1$, $\ldots$, $K_n$ have the same non-diagonal part, then 
\be 
\rho (\alpha_1 K_1 + \cdots + \alpha_n K_n) \le \alpha _1 \rho (K_1)+ \cdots + \alpha _n \rho (K_n) . 
\label{same_non-diag}
\ee
In other words, if $D_1, \ldots , D_n$ are diagonal matrices and $K$ a matrix such that $K+D_1$, $\ldots$, $K+D_n$
are nonnegative matrices that define operators on $L$, then we have   
\be
\rho \left(\alpha _1 (K+D_1)+\cdots +\alpha _n (K+D_n)\right) \le \alpha _1 \rho (K+D_1)+ \cdots + \alpha _n \rho (K+D_n).
\label{diagCohen}
\ee
If, in addition,  $L$ and $L^*$ have order continuous norms then under the above conditions inequalities (\ref{ckh}), (\ref{same_non-diag}) and (\ref{diagCohen}) hold also for $\rho =r_{ess}$.
\label{cohen}
\end{theorem}

Recall also the following well known inequality (see e.g. \cite{P06}, \cite{DP05}, \cite{H07}) for nonnegative measurable functions  and for $\alpha$ and $\beta$ nonnegative numbers such that $\alpha + \beta \ge 1$: 
 \be
f_1 ^{\alpha}g_1 ^{\beta} + \ldots +f_m ^{\alpha}g_m ^{\beta} \le (f_1+ \ldots + f_m)^{\alpha}(g_1+ \ldots + g_m)^{\beta}.
\label{H}
\ee

The rest of the article is organized as follows. In Section 2 we first state Theorem \ref{severalH} which we will need in our proofs and follows directly from results of \cite{P21a}. Then we prove new results on geometric symmetrizations of positive kernel operators and their weighted versions, which generalize several results from  \cite{Drn19}, \cite{P21a}, \cite{DP05},  \cite{P06} and \cite{H07} and we prove new versions of these results. We conclude the article with Section 3, where we establish some new additional results on Hadamard weighted geometric means of operators. In particular, in Section 3 we extend the main results of \cite{Z18} and some results of \cite{P18a}.

\section{Results on weighted geometric symmetrizations}

Given a nonnegative matrix $K$ that defines an operator on $L$ in $\mathcal{L}$, 
let us denote $\|K\|_{\infty}= \sup _{i,j \in R} k(i,j)$. 
Then we have $\|K\|_{\infty} \le \|K\|$ by (\ref{lg}).

In the following result 
we state versions of  
(\ref{very_good}). The result follows  from Inequalities (\ref{basic2}), 
and (\ref{very_good}) combined with \cite[Theorem 2.6]{P21a} applied to the matrix $H$ from (\ref{citat_H}). 
\begin{theorem}
Given $L$ in $\mathcal{L}$, 
let $\{K_{i j}\}_{i=1, j=1}^{l, n}$ be
nonnegative matrices that define operators on $L$ and
$\alpha _1$, $\alpha _2$,..., $\alpha _n$ positive numbers such that
$s_n =\sum_{i=1}^n \alpha _i \ge 1$. Let 
\begin{eqnarray} 
\label{citat_H}
H:&= &\left(K_{1 1}^{(\alpha _1)} \circ \cdots \circ K_{1 n}^{(\alpha _n)}\right) \ldots \left(K_{l 1}^{(\alpha _1)} \circ \cdots \circ K_{l n}^{(\alpha _n)} \right),\\
\nonumber
H_i: &= &K_{1 i} \cdots  K_{l i}, \;\;\;  M = \max _{i=1,\ldots, n} \|H_i\|_{\infty}, \\
\nonumber
  \beta&=&M^{s_n-1} \ \  \mathrm{and} \ \ \beta _i = \frac{\alpha _i}{ s_n} \ \  \mathrm{for}  \ \ \mathrm{ all}  \ \  i=1,\ldots, n . 
\end{eqnarray}
Then inequalities
\begin{eqnarray}
\nonumber
H & \le & H_1^{(\alpha_1)} \circ \cdots \circ H_n^{(\alpha_n)} \le \beta  H_1^{(\beta_1)} \circ \cdots \circ H_n^{(\beta_n)} \\
\nonumber
& \le & \beta C(H_1,\ldots , H_n, \beta_1, \ldots , \beta _n) \le \beta (\beta _1 H_1 + \cdots + \beta _n H_n), \\
\nonumber
\rho  \left(H \right)   &\le &  
\rho \left( H_1 ^{(\alpha _1)} \circ \cdots 
\circ H_n ^{(\alpha _n)} \right)  \le \beta \rho (  H_1^{(\beta_1)} \circ \cdots \circ H_n^{(\beta_n)} ) 
\label{very_good2}
\\
&\le &   
\nonumber
 \beta  \rho \left( H_1 \right)^{\beta _1} \cdots 
\rho \left( H_n \right)^{\beta _n} \le
\beta (\beta _1 \rho ( H_1) + \cdots + \beta _n \rho ( H_n ) ) \\
& \le &  \beta (\alpha _1 \rho ( H_1 ) + \cdots + \alpha _n \rho ( H_n ) ) ,
\\
\nonumber
\rho \left(H \right)   &\le &  \rho (H_1^{(\alpha_1)} \circ \cdots \circ H_n^{(\alpha_n)} ) 
\le  \beta \rho (  H_1^{(\beta_1)} \circ \cdots \circ H_n^{(\beta_n)} )\\
\label{new_geom_2_H}
&\le &   \beta  \rho( C(H_1,\ldots , H_n, \beta_1, \ldots , \beta _n)) \le  \beta (\beta _1 \rho(H_1)+ \cdots + \beta _n \rho (H_n))  \\
\nonumber
\mathit{hold}\;\mathit{for}\;  \mathit{all} \; &\rho \in \{r, \|\cdot\| \}&\; \mathit{and} \; \mathit{inequalities}\\
\nonumber
d(H )   &\le &  
d( H_1 ^{(\alpha _1)} \circ \cdots \circ H_n ^{(\alpha _n)} ) \le   \beta  d( H_1^{(\beta_1)} \circ \cdots \circ H_n^{(\beta_n)} )\\
\nonumber
 &\le &  \beta d(C(H_1,\ldots , H_n, \beta_1, \ldots , \beta _n) )  
\le     \beta d( \beta _1 H_1 + \cdots + \beta _n H_n ) \\
\label{new_geom_norm2_H}
& \le  &\beta (\beta _1 d(H_1)+ \cdots + \beta _n d(H_n))
\end{eqnarray}
hold for $d=\|\cdot \|$.

If, in addition, $L$ and $L^*$ have order continuous norms then inequalities (\ref{very_good2}) and (\ref{new_geom_2_H}) hold also for all $\rho \in \{r_{ess},\gamma \}$ and 
inequalities  (\ref{new_geom_norm2_H}) 
hold also for 
$d=\gamma$.

If, in addition, $L=l^2 (R)$, then inequalities (\ref{very_good2}) and (\ref{new_geom_norm2_H}) (and (\ref{new_geom_2_H})) hold also for $\rho =w$ and for $d=w$.

If, in addition, the matrices $H_1, \ldots, H_n$ are $m\times m$ matrices and the diagonal part of  $H_1^{(\alpha_1)} \circ  \cdots \circ H_n^{(\alpha_n)}$ is equal to zero, then
\be
r \left(H \right)   \le  r \left( H_1 ^{(\alpha _1)} \circ \cdots \circ H_n ^{(\alpha _n)} \right)  \le (m-1)\delta,
\label{mm}
\ee
where $\delta =\max \{M^{s_n}, 1\}$.

\label{severalH}
\end{theorem}

Let us recall 
the notion of geometric symmetrization of positive kernel operators 
on $L^2(X,\mu) $.
  Let $K$ be a positive kernel operator on $L^2(X,\mu)$ with kernel $k$.
The geometric symmetrization $S(K)$ of $K$ is the 
positive {\bf selfadjoint} kernel operator on $L^2(X,\mu)$ with kernel equal to $\sqrt{k(x,y)k(y,x)}$
at point $(x, y) \in X \times X$. Note that $S(K)=K^{(1/2)} \circ (K^*)^{(1/2)}$, 
since the kernel of the adjoint operator $K^*$ is equal to $k(y,x)$ at point $(x, y) \in X \times X$.

Next we extend several results  from \cite{Drn19}, \cite{P21a}, \cite{DP05}, \cite{P06} and \cite{H07} to the setting of weighted geometric ``symmetrizations'' $S_{\alpha} (\cdot) $ of positive kernel operators and prove new related results to those from the above references. Let $K$ be a positive kernel operator on $L^2 (X, \mu)$ and $\alpha \in [0,1]$. Denote 
$S_{\alpha} (K) = K^{(\alpha )} \circ (K^*)^{(1-\alpha)} $, which is a kernel operator on $L^2 (X, \mu)$ with a kernel $k^{\alpha}(x,y)k^{1-\alpha}(y,x)$. Observe that  
$(S_{\alpha} (K))^* =S_{\alpha} (K^*)= S_{1-\alpha} (K) $.


The following result generalizes and refines \cite[Propositions 3.1 and 3.2]{P21a}. 
\begin{proposition} Let $K, K_1 ,\ldots, K_n$ be positive kernel operators on $L^2(X,\mu)$ and $\alpha \in [0,1]$. Then  we have

$$\rho (S_{\alpha} (K_{1}) \cdots  S_{\alpha}(K_{n}))\;\;\;\;\;\;\;\;\;\;\;\;\;\;\;\;\;\;\;\;\;\;\;\;\;\;\;\;\;\;\;\;\;\;\;\;\;\;\;\;\;\;\;\;\;\;\;\;\;\;\;\;\;\;\;\;\;\;\;\;\;\;\;\;\;\;\;\;$$
\be \le \rho \left((K_1 \cdots K_n )^{(\alpha)} \circ ((K_n \cdots K_1)^*)^{(1-\alpha)}  \right) \le \rho (K_1 \cdots K_n )^{\alpha} \, \rho (K_n \cdots K_1 )^{1-\alpha} ,
\label{glavnaSN_rho}
\ee 
\be
\rho(S_{\alpha}(K_1)+ \cdots + S_{\alpha}(K_m)) \le \rho\left(S_{\alpha}(K_1+ \cdots+ K_m)\right) \le \rho(K_1+ \cdots+ K_m)
\label{+several}
\ee
for all $\rho \in \{ r, r_{ess}, \gamma, \|\cdot\|, w \}$.
In particular, for all $\rho \in \{ r, r_{ess}, \gamma, \|\cdot\|, w \}$ we have
\be
\rho\left(S_{\alpha}(K)\right) \le \rho (K) . 
\label{Schwenk_rho}
\ee 
We also have
\be
\rho \left(S_{\alpha}(K_{1}) S_{\alpha}(K_{2})\right) \le \rho \left((K_1 K_2 )^{(\alpha)} \circ ((K_2 K_1)^*)^{(1-\alpha)}  \right) \le  \rho (K_1 K_2) .
\label{radius_two_rho_first}
\ee 
for  $\rho \in \{ r, r_{ess}\}$.
\label{weig_sym}
\end{proposition}
\begin{proof}
By (\ref{very_good}) 
we have 
$$\rho \left(S_{\alpha} (K_{1}) \cdots  S_{\alpha}(K_{n})\right) = \rho \left( (K_1 ^{(\alpha)} \circ (K_1^*)^{(1-\alpha)}) \cdots 
\left(K_n ^{(\alpha)} \circ (K_n^*)^{(1-\alpha)}\right)\right) $$
$$\le \rho \left((K_1 \cdots K_n )^{(\alpha)} \circ ((K_n \cdots K_1)^*)^{(1-\alpha)}  \right)$$
$$ \le \rho (K_1 \cdots K_n)^{\alpha} \, \rho ((K_n \cdots K_1)^*)^{1-\alpha} = 
\rho (K_1 \cdots K_n )^{\alpha} \, \rho (K_n \cdots K_1 )^{1-\alpha}.$$
This proves (\ref{glavnaSN_rho}). 
Inequality  (\ref{Schwenk_rho}) is a special cases of (\ref{glavnaSN_rho}) 
while (\ref{radius_two_rho_first}) follows from (\ref{glavnaSN_rho}) and
$\rho(K_1 K_2)=\rho(K_2 K_1)$ for $\rho \in \{ r, r_{ess}\}$. 

Inequalities (\ref{+several}) 
follow from (\ref{H}) and  (\ref{Schwenk_rho}).
\end{proof}
If $K$ is a nonnegative matrix that defines an operator on $l^2 (R)$ and if $\alpha$ and $\beta$ are nonnegative numbers such that $\alpha + \beta \ge 1$, then a nonnegative matrix \\
$S_{\alpha, \beta} (K) = K^{(\alpha)} \circ (K^*)^{(\beta)}  $ also defines an operator on $l^2 (R)$ by Theorem \ref{good_work}(ii). The following result is proved in a similar way as Proposition \ref{weig_sym}.

\begin{proposition} Let $K, K_1 ,\ldots, K_n$ be nonnegative matrices that define operators on $l^2 (R)$ and let $\alpha$ and $\beta$ be nonnegative numbers such that $\alpha + \beta \ge 1$. Then  we have
$$\rho (S_{\alpha ,\beta} (K_{1}) \cdots  S_{\alpha  ,\beta}(K_{n})) \;\;\;\;\;\; \;\;\;\;\;\;\;\;\;\;\;\;\;\;\;\;\;\;\;\;\;\;\;\;\;\;\;\;\;\;\;\;\;\;\;\;\;\;\;\;\;\;\;\;\;\;\;\;\;\;\;\;\;\;$$
\be \le \rho \left((K_1 \cdots K_n )^{(\alpha)} \circ ((K_n \cdots K_1)^*)^{(\beta)}  \right) 
\le \rho (K_1 \cdots K_n )^{\alpha} \, \rho (K_n \cdots K_1 )^{\beta},
\label{glavnaSN_rho_s}
\ee 
\be
\rho\left(S_{\alpha ,\beta}(K)\right) \le \rho (K)^{\alpha + \beta},  
\label{Schwenk_rho_s}
\ee 
\be
\rho(S_{\alpha  ,\beta}(K_1)+ \cdots + S_{\alpha  ,\beta}(K_m)) \le \rho\left(S_{\alpha  ,\beta}(K_1+ \cdots+ K_m)\right) \le \rho(K_1+ \cdots+ K_m)^{\alpha + \beta}
\label{+several_s}
\ee

for all $\rho \in \{ r, \|\cdot\|\}$. We also have
\be
\rho\left(S_{\alpha ,\beta}(K_{1}) S_{\alpha ,\beta}(K_{2})\right) \le \rho(K_1 K_2)^{\alpha + \beta} 
\label{radius_two_rho_s}
\ee 
for $\rho=r$. Moreover, we have
\be
\rho (S_{\alpha ,\beta} (K_{1}) \cdots  S_{\alpha  ,\beta}(K_{n}))  \le \rho \left((K_1 \cdots K_n )^{(\alpha)} \circ ((K_n \cdots K_1)^*)^{(\beta)}  \right) 
\label{glavnaSN_rho_s2}
\ee
$$ \le \delta \rho \left((K_1 \cdots K_n )^{(\frac{\alpha}{\alpha + \beta})} \circ ((K_n \cdots K_1)^*)^{(\frac{\beta}{\alpha + \beta})}  \right) 
\le \delta \cdot \rho (K_1 \cdots K_n )^{\frac{\alpha}{\alpha +\beta}} \, \rho (K_n \cdots K_1 )^{\frac{\beta}{\alpha +\beta}} ,$$
where $\delta = \max\{\|K_1 \cdots K_n\|_{\infty}, \|K_n \cdots K_1\|_{\infty} \}^{\alpha + \beta -1}$, and

\be
\rho\left(S_{\alpha ,\beta}(K)\right) \le \|K\|_{\infty} ^{\alpha + \beta -1} \rho\left(S_{\frac{\alpha}{\alpha +\beta}}(K)\right) \le \|K\|_{\infty} ^{\alpha + \beta -1}\rho (K),  
\label{Schwenk_rho_s2}
\ee 
\be
\rho(S_{\alpha  ,\beta}(K_1)+ \ldots + S_{\alpha  ,\beta}(K_m)) \le \rho\left(S_{\alpha  ,\beta}(K_1+ \ldots+ K_m)\right) 
\label{+several_s2}
\ee
$$\le \|K_1+ \ldots+ K_m\|_{\infty} ^{\alpha + \beta -1} \rho\left(S_{\frac{\alpha}{\alpha +\beta}}(K_1+ \ldots+ K_m)\right)\le \|K_1+ \ldots+ K_m\|_{\infty} ^{\alpha + \beta-1} \rho(K_1+ \ldots+ K_m)$$
for all $\rho \in \{ r, r_{ess}, \gamma, \|\cdot\|, w \}$.
We also have

\be
\rho\left(S_{\alpha ,\beta}(K_{1}) S_{\alpha ,\beta}(K_{2})\right) \le \rho \left((K_1 K_2 )^{(\alpha)} \circ ((K_2 K_1)^*)^{(\beta)}  \right) 
\label{radius_two_rho_s2}
\ee
$$ \le \max\{\|K_1 K_2\|_{\infty}, \|K_2K_1\|_{\infty} \}^{\alpha + \beta -1} \rho \left((K_1 K_2 )^{(\frac{\alpha}{\alpha + \beta})} \circ ((K_2 K_1)^*)^{(\frac{\beta}{\alpha + \beta})}  \right)$$
$$\le \max\{\|K_1 K_2\|_{\infty}, \|K_2K_1\|_{\infty} \}^{\alpha + \beta -1}\rho(K_1 K_2)  \;\;\;\;\;\; \;\;\;\;\;\;\;\;\;\;\;\;\;\;\;\;\;\;\;\;\;\;\;\;\;\;\;\;\;\;\;\;\;$$
for all $\rho \in \{ r, r_{ess}\}$.
\end{proposition}
\begin{proof} Inequalities (\ref{glavnaSN_rho_s}) are proved in similar way as  Inequalities (\ref{glavnaSN_rho})  by applying  Theorem \ref{good_work}(ii).  
Inequalities (\ref{glavnaSN_rho_s2}) follow from (\ref{very_good2}). Inequalities (\ref{Schwenk_rho_s}) and (\ref{Schwenk_rho_s2}) are special cases of   (\ref{glavnaSN_rho_s}) and  (\ref{glavnaSN_rho_s2}), respectively. 
Inequalities (\ref{+several_s}) and (\ref{+several_s2}) follow from (\ref{H}),  (\ref{Schwenk_rho_s}) and (\ref{Schwenk_rho_s2}), while Inequalities 
(\ref{radius_two_rho_s}) and (\ref{radius_two_rho_s2}) follow from (\ref{glavnaSN_rho_s}) and  (\ref{glavnaSN_rho_s2}).
\end{proof}

The following two results generalize \cite[Lemma 2.1]{Drn19}, \cite[Theorem 2.2]{Drn19} and \cite[Theorem 3.5]{P21a} by employing a similar but more general method of proof. 
\begin{lemma} (i) If $K$ is a positive kernel operator on $L^2(X,\mu)$ and $\alpha \in [0,1]$, then
\be
S_{\alpha}(K^2) \ge S_{\alpha}(K)^2.
\label{Roman_g}
\ee
(ii) If $K$ is a nonnegative matrix that defines an operator on $l^2 (R)$ and if $\alpha$ and $\beta$ are nonnegative numbers such that $\alpha + \beta \ge 1$, then 
\be
S_{\alpha, \beta}(K^2) \ge S_{\alpha , \beta}(K)^2.
\label{Roman_ge}
\ee
\end{lemma}
\begin{proof} The kernel of $S_{\alpha}(K^2)$ at $(x,y) \in X \times X$ equals
$$\left(\int _X k(x,z)k(z,y) d\mu (z)\right)^{\alpha} \left(\int _X k(y,z)k(z,x) d\mu (z)\right)^{1-\alpha}. $$
By Hoelder's inequality this is larger or equal to 
$$\int _X (k(x,z)k(z,y))^{\alpha} (k(y,z)k(z,x) )^{1-\alpha}d\mu (z) $$
$$= \int _X k(x,z)^{\alpha}k(z,x)^{1-\alpha}k(z,y)^{\alpha} k(y,z)^{1-\alpha}d\mu (z)$$
and this equals the kernel of $S_{\alpha}(K)^2$ at $(x,y)$, which proves (\ref{Roman_g}).

Inequality (\ref{Roman_ge}) is proved in a similar way by \cite[Proposition 4.1]{P06}. 
\end{proof}

\begin{theorem} (i) Let $K$ be a positive kernel operator on $L^2(X,\mu)$, $\alpha \in [0,1]$ and let 
$\rho_n = \rho (S_{\alpha}( K ^{2^n}))^{2^{-n}} $ for $n \in \mathbb{N}\cup \{0\}$ and $ \rho \in \{r, r_{ess}\}$. Then for each $n$
$$ \rho (S_{\alpha}(K))= \rho_0 \le \rho_1 \le \cdots \le \rho_n \le \rho (K).$$

(ii) Let $K$ be a nonnegative matrix that defines an operator on $l^2 (R)$ and $\alpha$ and $\beta$ nonnegative numbers such that $\alpha + \beta \ge 1$. If 
$r_n = \rho (S_{\alpha, \beta}( K ^{2^n}))^{2^{-n}} $ for $n \in \mathbb{N}\cup \{0\}$ and $\rho \in \{r, r_{ess}\}$, then 
$$ \rho (S_{\alpha, \beta}(K))= r_0 \le r_1 \le \cdots \le r_n \le \min\{ \rho (K)^{\alpha +\beta}, \|K^{2^n}\|_{\infty} ^{\frac{\alpha + \beta -1}{2^n}}\rho (K)\} \;\mathrm{for}\;
\rho =r,$$
$$ \rho (S_{\alpha, \beta}(K))= r_0 \le r_1 \le \cdots \le r_n \le  \|K^{2^n}\|_{\infty} ^{\frac{\alpha + \beta -1}{2^n}}\rho (K)\; \;\mathrm{for}\;
\rho =r_{ess} \;\; \mathrm{and}$$
$$ r_n \le \|K^{2^n}\|_{\infty} ^{\frac{\alpha + \beta -1}{2^n}} \rho\left(S_{\frac{\alpha}{\alpha +\beta}}(K^{2^n})\right)^{2^{-n}} \le \|K^{2^n}\|_{\infty} ^{\frac{\alpha + \beta -1}{2^n}}\rho (K).$$
\end{theorem}
\begin{proof} To prove (i) we first observe 
 that by (\ref{Roman_g}) (or by  (\ref{radius_two_rho_first})) we have
\be
\rho (S_{\alpha}(K^2) )\ge \rho (S_{\alpha}(K)^2) = \rho (S_{\alpha}(K))^2.
\label{Roman2}
\ee
By  (\ref{Schwenk_rho}) $\rho (S_{\alpha}(K ^{2^n})) \le \rho (K ^{2^n}) =   \rho (K) ^{2^n}$ and so $\rho_n \le \rho (K)$. Since $\rho_{n-1} \le \rho_{n}$ for all $n\in \mathbb{N}$ by (\ref{Roman2}) the proof of (i) is completed.

In a similar way (ii) is proved by applying  (\ref{Schwenk_rho_s}),  (\ref{Schwenk_rho_s2}) and (\ref{radius_two_rho_s2}).
\end{proof}
The following result generalizes and extends \cite[Theorem 2.2 and Theorem 3.2 (3)]{H07}.
\begin{proposition}
Let $K$ be a positive kernel operators on $L^2(X,\mu)$ and $\alpha \in [0,1]$. Then 
for all $\rho \in \{ r, r_{ess}, \gamma, \|\cdot\|, w \}$ and $n\in \mathbb{N}$ we have
\be
\rho\left(S (K)\right) \le \rho\left(S_{\alpha}(K)\right) \le \rho (K)  \;\;\mathrm{and} 
\label{Schwenk_rho_good}
\ee 
\be
\rho\left(S (K^n)\right)^{\frac{1}{n}} \le \rho\left(S_{\alpha}(K^n)\right)^{\frac{1}{n}} \le \rho (K). 
\label{Schwenk_rho_good2}
\ee 
\end{proposition}
\begin{proof} Since $S(K)=S(S_{\alpha}(K))$ Inequalities (\ref{Schwenk_rho_good}) follow from (\ref{Schwenk_rho}). Inequalities (\ref{Schwenk_rho_good2}) follow from (\ref{Schwenk_rho_good}).
\end{proof}
The following result generalizes and extends \cite[Theorems 2.3 and 3.3]{H07}. It is proved in similar way as \cite[Theorem 2.3]{H07} by applying (\ref{Schwenk_rho}). To avoid too much repetition of ideas we omit the details of the proof.
\begin{theorem} Let $K$ be a positive kernel operators on $L^2(X,\mu)$. \\
For $\rho \in \{ r, r_{ess}, \gamma, \|\cdot\|, w \}$ and $\alpha \in [0,1]$ define $f_{\rho} (\alpha)= \rho (S_{\alpha}(K))$. Then $f_{\rho}$ is decreasing in $[0, 0.5]$ and increasing in $[0.5, 1]$.
\end{theorem}
\section{Additional results on weighted geometric means}

The following refinement of inequality (\ref{gl1vecr}) was proved in \cite[Corollary 3.10]{P18b}.

\begin{theorem}
Let $K_1, \ldots ,K _n$ be positive kernel operators on a Banach function space $L$. 
If 
 $\alpha _1, \ldots ,\alpha _n$  are positive numbers such that 
$\sum _{i=1} ^n \alpha _i = 1$ and if  $m \in \NN$ then
\be
\rho (K_1 ^{( \alpha _1)} \circ \cdots \circ K_n ^{(\alpha _n)} ) \le \rho ((K_1 ^m ) ^{( \alpha _1)} \circ \cdots \circ (K_n ^m) ^{(\alpha _n)} ) ^{ \frac{1}{m}} \le
\rho (K_1)^{ \alpha _1} \, \cdots \rho (K_n)^{\alpha _n} 
\label{rad_ref}
\ee
for $\rho =r$.

If, in addition, $L$ and $L^*$ have order continuous norms then Inequalities (\ref{rad_ref}) hold also for $\rho=r_{ess}$.

\label{ref_powers}
\end{theorem}

By iterating  (\ref{rad_ref}) we obtain its refinement.

\begin{corollary}
Let $K_1, \ldots ,K_n$ be positive kernel operators on a Banach function space $L$. 
If 
 $\alpha _1, \ldots ,\alpha _n$  are positive numbers such that 
$\sum _{i=1} ^n \alpha _i = 1$ and if  $m, l \in \NN$ then
\begin{eqnarray}
\nonumber
& & \rho (K_1 ^{( \alpha _1)} \circ \cdots \circ K_n ^{(\alpha _n)} )  \le  \rho  ((K_1 ^m ) ^{( \alpha _1)} \circ \cdots \circ (K_n ^m)  ^{(\alpha _n)} ) ^{ \frac{1}{m}} \\
&\le &  \rho ((K_1 ^{ml} ) ^{( \alpha _1)} \circ \cdots \circ (K_n ^{ml})  ^{(\alpha _n)} ) ^{ \frac{1}{ml}} \le
\rho (K_1)^{ \alpha _1} \, \cdots \rho (K_n)^{\alpha _n} 
\label{rad_ref_kl}
\end{eqnarray}
for $\rho =r$.

If, in addition, $L$ and $L^*$ have order continuous norms then Inequalities (\ref{rad_ref_kl}) hold also for $\rho=r_{ess}$.

\label{ref_powers_kl}
\end{corollary}
The following result 
follows  from (\ref{rad_ref_kl}) and  Theorem \ref{severalH}. 
\begin{corollary}
Given $L$ in $\mathcal{L}$, let $K_1, \ldots , K_n$ be nonnegative matrices that define operators on $L$ and
$\alpha_1, \ldots, \alpha _n$ nonnegative numbers such that $s_n = \sum _{i=1} ^n \alpha _i \ge 1$. Let
$$ M = \max _{i=1,\ldots, n} \|K_i\|_{\infty},   \;\; \\ \\\\ \beta=M^{s_n-1} $$
and  $\beta _i =\frac{\alpha _i}{s_n}$ for $i=1, \ldots, n$.

Then
\begin{eqnarray}
\nonumber
 \rho (K_1 ^{( \alpha _1)} \circ \cdots \circ K_n ^{(\alpha _n)} )  &\le &  
\beta \rho(  K_1^{(\beta_1)} \circ \cdots \circ K_n^{(\beta_n)} ) \le 
\beta \rho ((K_1 ^m ) ^{( \beta _1)} \circ \cdots \circ (K_n ^m)  ^{(\beta _n)} ) ^{ \frac{1}{m}} \\
&\le & \beta  \rho ((K_1 ^{ml} ) ^{( \beta _1)} \circ \cdots \circ (K_n ^{ml})  ^{(\beta _n)} ) ^{ \frac{1}{ml}} \le
\beta \rho (K_1)^{ \beta _1} \, \cdots \rho (K_n)^{\beta _n} 
\label{new_geom_ref1}
\end{eqnarray}

for all $m, l \in \NN$ and $\rho=r$.

If, in addition, $L$ and $L^*$ have order continuous norms then Inequalities (\ref{new_geom_ref1}) hold also for $\rho=r_{ess}$.
\label{useful_cor}
\end{corollary}
Next we prove (with a standard method from e.g. \cite{DP05} and \cite{P06}) that in the 
case of sequence spaces $L\in \mathcal{L}$ 
 inequalities (\ref{rad_ref_kl}) for the spectral radius hold  also under the condition 
$\sum _{i=1} ^n \alpha _i \ge 1$. In this case we also prove additional refinements of (\ref{rad_ref}).
\begin{theorem} 
\label{ugly_ref}
Given $L\in \mathcal{L}$, let 
$K_1, \ldots ,K _n$ be nonnegative matrices that define operators on $L$.  If 
 $\alpha _1, \ldots ,\alpha _n$  are nonnegative numbers such that 
$s_n = \sum _{i=1} ^n \alpha _i \ge 1$ and if  $m, l \in \NN$ and $\beta _i = \frac{\alpha _i}{ s_n}$ for all $i=1,\ldots, n $, then we have 
\begin{eqnarray}
\nonumber
& & r (K_1 ^{( \alpha _1)} \circ \cdots \circ K_n ^{(\alpha _n)} )  \le  r ((K_1 ^m ) ^{( \alpha _1)} \circ \cdots \circ (K_n ^m)  ^{(\alpha _n)} ) ^{ \frac{1}{m}} \\
\nonumber
&\le &  r ((K_1 ^{ml} ) ^{( \alpha _1)} \circ \cdots \circ (K_n ^{ml})  ^{(\alpha _n)} ) ^{ \frac{1}{ml}} \le 
r ((K_1 ^{m l}) ^{(\beta _1)} \circ \cdots \circ (K_n ^{ml}) ^{(\beta _n)} ) ^{ \frac{s_n}{ml}} \\
& \le &
r(K_1)^{ \alpha _1} \, \cdots r(K_n)^{\alpha _n} 
\label{rad_ref_kl_ref_sn}
\end{eqnarray}
and
\begin{eqnarray}
\nonumber
& &r (K_1 ^{( \alpha _1)} \circ \cdots \circ K_n ^{(\alpha _n)} ) \le  r ((K_1 ^m ) ^{( \alpha _1)} \circ \cdots \circ (K_n ^m) ^{(\alpha _n)} ) ^{ \frac{1}{m}}  \\
\nonumber
&\le &r ((K_1 ^m ) ^{(\beta _1)} \circ \cdots \circ (K_n ^m) ^{(\beta _n)} ) ^{ \frac{s_n}{m}} 
\le 
r ((K_1 ^{m l}) ^{(\beta _1)} \circ \cdots \circ (K_n ^{ml}) ^{(\beta _n)} ) ^{ \frac{s_n}{ml}} \\
&\le &
r(K_1)^{ \alpha _1} \, \cdots r(K_n)^{\alpha _n} .
\label{rad_ref_sm}
\end{eqnarray}
\end{theorem}
\begin{proof} First we prove that (\ref{rad_ref_kl}) holds also under our assumptions.
By (\ref{basic2}) we have
\begin{eqnarray}
\nonumber
& &\left(K_1 ^{( \alpha _1)} \circ \cdots \circ K_n ^{(\alpha _n)} \right)^m = \left(K_1 ^{( \alpha _1)} \circ \cdots \circ K_n ^{(\alpha _n)} \right) \cdots \left(K_1 ^{( \alpha _1)} \circ \cdots \circ K_n ^{(\alpha _n)} \right) \\
& \le &  (K_1 ^m ) ^{( \alpha _1)} \circ \cdots \circ (K_n ^m) ^{(\alpha _n)} .
\label{Roman_good}
\end{eqnarray}
It follows from (\ref{Roman_good}) and  (\ref{gl1vecr}) that 
$$r (K_1 ^{( \alpha _1)} \circ \cdots \circ K_n ^{(\alpha _n)} )^m = r \left (\left(K_1 ^{( \alpha _1)} \circ \cdots \circ K_n ^{(\alpha _n)} \right)^m \right)$$
$$ \le  r ((K_1 ^m ) ^{( \alpha _1)} \circ \cdots \circ (K_n ^m) ^{(\alpha _n)} ) \le 
r(K_1 ^m)^{ \alpha _1} \, \cdots r(K_n ^m)^{\alpha _n}=(r(K_1)^{ \alpha _1} \, \cdots r(K_n)^{\alpha _n})^m  $$
which proves (\ref{rad_ref}) in this case. By iterating as before one obtains (\ref{rad_ref_kl}) under our assumptions.

Let us prove (\ref{rad_ref_sm}). 
Since $s_n \ge 1$ and
$\alpha _i = \beta _i s_n$ it follows by the first inequality in  (\ref{rad_ref_kl}) in the case $s_n \ge 1$, (\ref{gl3tr}), (\ref{gl1vecr}) and  (\ref{rad_ref_kl}) that
\begin{eqnarray}
\nonumber
& &r (K_1 ^{( \alpha _1)} \circ \cdots \circ K_n ^{(\alpha _n)} ) \le  r ((K_1 ^m ) ^{( \alpha _1)} \circ \cdots \circ (K_n ^m) ^{(\alpha _n)} ) ^{ \frac{1}{m}}  \\
\nonumber
&= &  r \left( \left( (K_1 ^m ) ^{( \beta _1)} \circ \cdots \circ (K_n ^m) ^{(\beta _n)} \right)^{(s_n)} \right) ^{ \frac{1}{m}}\le
r ((K_1 ^{m }) ^{(\beta _1)} \circ \cdots \circ (K_n ^{m}) ^{(\beta _n)} ) ^{ \frac{s_n}{m}} \\
\nonumber
&\le &
r ((K_1 ^{ml }) ^{( \beta_1 )} \circ \cdots \circ (K_n ^{ml}) ^{(\beta_n )} ) ^{ \frac{s_n}{ml}} 
\le  ( r (K_1 ^{ml } ) ^{ \alpha_1 } \cdots  r(K_n ^{ml }) ^{\alpha_n} ) ^{ \frac{1}{ml}} \\
&=&
r(K_1)^{ \alpha _1} \, \cdots r(K_n)^{\alpha _n} ,
\nonumber
\end{eqnarray}
which 
proves  (\ref{rad_ref_sm}). Now (\ref{rad_ref_kl_ref_sn}) follows from  (\ref{rad_ref_kl}) in the case $s_n \ge 1$ and (\ref{rad_ref_sm}), which completes the proof.
\end{proof}

We conclude the article by extending the main results of \cite{Z18} and some results of \cite{P18a}. The following result is a new variation of \cite[Theorem 4.1]{P18a} for even $m$. By $\sigma _m$ we denote the group of permutations of the set $\{1, \ldots , m\}$.

\begin{theorem} Let $m$ be even, $\{\tau, \nu\} \subset \sigma_m$ and let
$ H_1, \ldots , H_m$ be positive kernel operators on $ L^2(X,\mu)$. For $j=1, \ldots, \frac{m}{2} $ denote $A_j =H_{\tau{(2j-1)}} ^*H_{\tau{(2j)}}$ and $A_{\frac{m}{2} +j} =A_j ^*=H_{\tau{(2j)}} ^*H_{\tau{(2j-1)}}$. Let $P_i = A_{\nu (i)} \cdots A_{\nu (m)} A_{\nu (1)} \cdots A_{\nu (i-1)}$ for $i=1, \ldots , m$. 

\noindent (i) Then 
$$\|H_1^{(\frac{1}{m})}\circ\cdots \circ H_m^{(\frac{1}{m})}\|\le r(A_1^{(\frac{1}{m})}\circ\cdots \circ A_m^{(\frac{1}{m})})^{\frac{1}{2}}$$

\be
\le r\left(P_1^{(\frac{1}{m})}\circ P_2^{(\frac{1}{m})}\circ \cdots \circ P_m^{(\frac{1}{m})}\right)^{\frac{1}{2m}}\le   r\left( A_{\nu (1)} \cdots A_{\nu (m)}\right)^{\frac{1}{2m}}.
\label{KA1}
\ee

\label{KA}

\noindent (ii) If $ H_1, \ldots , H_m$ are nonnegative matrices that define operators on $l^2 (R)$ and if $\alpha \ge \frac{1}{m}$, then

$$\|H_1^{(\alpha)}\circ\cdots \circ H_m ^{(\alpha)}\|\le r(A_1^{(\alpha)}\circ\cdots \circ A_m^{(\alpha)})^{\frac{1}{2}}$$

\be
\le r\left(P_1^{(\alpha)}\circ P_2^{(\alpha)}\circ \cdots \circ P_m^{(\alpha)}\right)^{\frac{1}{2m}}\le   r\left( A_{\nu (1)} \cdots A_{\nu (m)}\right)^{\frac{\alpha}{2}}.
\label{KA2}
\ee
\end{theorem}

\begin{proof} First we prove (\ref{KA1}).
By 
\be
\|H\|=r(H^* H)^{\frac{1}{2}}=r(H H^*)^{\frac{1}{2}},
\label{L2}
\ee
(\ref{very_good}) and commutativity of Hadamard product we have
$$\|H_1^{(\frac{1}{m})}\circ\cdots \circ H_m^{(\frac{1}{m})}\|=r((H_1^{(\frac{1}{m})}\circ\cdots \circ H_m^{(\frac{1}{m})})^*(H_1^{(\frac{1}{m})}\circ\cdots \circ H_m^{(\frac{1}{m})}))^{\frac{1}{2}}=$$
$$r[((H_{\tau{(1)}}^*)^{(\frac{1}{m})}\circ\cdots \circ(H_{\tau{(m-1)}}^*)^{(\frac{1}{m})}\circ(H_{\tau{(2)}}^*)^{(\frac{1}{m})}\circ\cdots \circ(H_{\tau{(m)}}^*)^{(\frac{1}{m})}))\cdot $$
$$(H_{\tau{(2)}}^{(\frac{1}{m})}\circ\cdots \circ H_{\tau{(m)}}^{(\frac{1}{m})}\circ H_{\tau{(1)}}^{(\frac{1}{m})}\circ\cdots \circ H_{\tau{(m-1)}}^{(\frac{1}{m})}))]^{\frac{1}{2}}$$
$$\le r((H_{\tau{(1)}}^*H_{\tau{(2)}})^{(\frac{1}{m})}\circ\cdots \circ(H_{\tau{(m-1)}}^*H_{\tau{(m)}})^{(\frac{1}{m})} \circ(H_{\tau{(2)}}^*H_{\tau{(1)}})^{(\frac{1}{m})}\circ\cdots \circ(H_{\tau{(m)}}^*H_{\tau{(m-1)}})^{(\frac{1}{m})}))^{\frac{1}{2}}$$
$$= r(A_1^{(\frac{1}{m})}\circ\cdots \circ A_m^{(\frac{1}{m})})^{\frac{1}{2}} =r(A_{\nu (1)}^{(\frac{1}{m})}\circ\cdots \circ A_{\nu (m)}^{(\frac{1}{m})})^{\frac{1}{2}},$$
which proves the first inequality in (\ref{KA}). The second and the third inequality in (\ref{KA1}) follow from   \cite[Inequalities (4.2)]{P18a}.

Inequalities (\ref{KA2}) are proved in a similar manner by applying Theorem \ref{good_work}(ii).

\end{proof}

By interchanging $H_i$ with $H_i ^*$ for all $i$ in Theorem \ref{KA} we obtain the following result.
\begin{corollary}
Let $m$ be even, $\tau \in \sigma_m$, $\beta \in [0,1]$ and let
$ H_1, \ldots , H_m$ be positive kernel operators on $ L^2(X,\mu)$. Let $A_j$ for $j=1, \ldots, m $ be as in Theorem \ref{KA} and denote  $B_j =H_{\tau{(2j-1)}} H^* _{\tau{(2j)}}$ and $B_{\frac{m}{2} +j} =B_j ^*=H_{\tau{(2j)}} H^* _{\tau{(2j-1)}}$ for $j=1, \ldots, \frac{m}{2} $. 

\noindent (i) Then 
$$\|H_1^{(\frac{1}{m})}\circ\cdots \circ H_m^{(\frac{1}{m})}\|\le r(B_1^{(\frac{1}{m})}\circ\cdots \circ B_m^{(\frac{1}{m})})^{\frac{1}{2}}$$
and 
$$\|H_1^{(\frac{1}{m})}\circ\cdots \circ H_m^{(\frac{1}{m})}\|\le  r(A_1^{(\frac{1}{m})}\circ\cdots \circ A_m^{(\frac{1}{m})})^{\frac{\beta}{2}}r(B_1^{(\frac{1}{m})}\circ\cdots \circ B_m^{(\frac{1}{m})})^{\frac{1-\beta}{2}}.$$

\noindent (ii) If $ H_1, \ldots , H_m$ are nonnegative matrices that define operators on $l^2 (R)$ and if $\alpha \ge \frac{1}{m}$, then
$$\|H_1 ^{(\alpha)}\circ\cdots \circ H_m^{(\alpha)}\|\le r(B_1 ^{(\alpha)}\circ\cdots \circ B_m ^{(\alpha)})^{\frac{1}{2}}$$
and 
$$\|H_1^{(\alpha)} \circ\cdots \circ H_m^{(\alpha)}\|\le  r(A_1^{(\alpha)}\circ\cdots \circ A_m^{(\alpha)})^{\frac{\beta}{2}}r(B_1 ^{(\alpha)}\circ\cdots \circ B_m ^{(\alpha)})^{\frac{1-\beta}{2}}.$$
\label{goodAB}
\end{corollary}

\begin{remark}{\rm In the special case of the identity permutation $\mu$ in Theorem \ref{KA} it holds
$$ r\left( A_{\mu (1)} \cdots A_{\mu (m)}\right)^{\frac{1}{2}}=\| H_{\tau{(1)}}^*H_{\tau{(2)}}H_{\tau{(3)}}^*H_{\tau{(4)}}\cdots H_{\tau{(m-1)}}^*H_{\tau{(m)}}\|$$
by (\ref{L2}).
}
\end{remark}


The following two results extend, generalize and refine \cite[Theorem 2.8]{Z18} and give an extension and a different refinement of \cite[Inequality (4.16)]{P18a} in the case $\alpha\ge\frac{2}{m}$. 
\begin{theorem}
 Let $m$ be even, $\alpha\ge\frac{2}{m}$, $\tau \in \sigma_m$ and
let $H_1,\ldots,H_m$ be nonnegative matrices that define operators on $l^2(R)$. Let $A_j$ for
$j=1, \ldots, m $ be as in Theorem \ref{KA} and denote $S_i =  A_{i} \cdots A_{\frac{m}{2}} A_{1} \cdots A_{i-1}$ for $i=1, \ldots , \frac{m}{2}$.  Then
$$\|H_1^{(\alpha)}\circ\cdots \circ H_m ^{(\alpha)}\|\le r(A_1^{(\alpha)}\circ\cdots \circ A_m^{(\alpha)})^{\frac{1}{2}} \le  r(A_1^{(\alpha)}\circ\cdots \circ A_{\frac{m}{2}}^{(\alpha)})$$
$$= r((H_{\tau{(1)}}^{*}H_{\tau{(2)}})^{(\alpha)}\circ( H_{\tau{(3)}}^{*}H_{\tau{(4)}})^{(\alpha)}\circ \cdots \circ(H_{\tau{(m-1)}}^{*}H_{\tau{(m)}})^{(\alpha)})$$
\be
\label{KA3}
\le r\left(S_1^{(\alpha)}\circ S_2^{(\alpha)}\circ \cdots \circ S_{\frac{m}{2}}^{(\alpha)}\right)^{\frac{2}{m}}\le r(H_{\tau{(1)}}^{*}H_{\tau{(2)}}H_{\tau{(3)}}^{*}H_{\tau{(4)}}\cdots H_{\tau{(m-1)}}^{*}H_{\tau{(m)}})^{\alpha}.
\ee
\end{theorem}
\begin{proof}
By the first inequality in (\ref{KA2}) and (\ref{gl1vecr}) in Theorem \ref{good_work}(ii) we have 
$$\|H_1^{(\alpha)}\circ\cdots \circ H_m ^{(\alpha)}\|\le r(A_1^{(\alpha)}\circ\cdots \circ A_m^{(\alpha)})^{\frac{1}{2}}$$
$$= r(A_1^{(\alpha)}\circ\cdots \circ A_{\frac{m}{2} }^{(\alpha)} \circ (A_1 ^*)^{(\alpha)}\circ\cdots \circ (A^*_{\frac{m}{2} } )^{(\alpha)})^{\frac{1}{2}}$$
$$\le (r(A_1^{(\alpha)}\circ\cdots \circ A_{\frac{m}{2} }^{(\alpha)})
 r ((A_1 ^{(\alpha)}\circ\cdots \circ A_{\frac{m}{2}} ^{(\alpha)})^*))^{\frac{1}{2}}= r(A_1^{(\alpha)}\circ\cdots \circ A_{\frac{m}{2}}^{(\alpha)})$$
$$= r((H_{\tau{(1)}}^{*}H_{\tau{(2)}})^{(\alpha)}\circ( H_{\tau{(3)}}^{*}H_{\tau{(4)}})^{(\alpha)}\circ \cdots \circ(H_{\tau{(m-1)}}^{*}H_{\tau{(m)}})^{(\alpha)}).$$

Since
$$((H_{\tau{(1)}}^{*}H_{\tau{(2)}})^{(\alpha)}\circ(H_{\tau{(3)}}^{*}H_{\tau{(4)}}) ^{(\alpha)}\circ\cdots \circ(H_{\tau{(m-1)}}^{*} H_{\tau{(m)}})^{(\alpha)})^{\frac{m}{2}}=$$
$$((H_{\tau{(1)}}^{*}H_{\tau{(2)}})^{(\alpha)}\circ\cdots \circ(H_{\tau{(m-1)}}^{*}H_{\tau{(m)}})^{(\alpha)})((H_{\tau{(3)}}^{*}H_{\tau{(4)}}) ^{(\alpha)}\circ\cdots\circ(H_{\tau{(1)}}^{*}H_{\tau{(2)}})^{(\alpha)})$$
$$\cdots((H_{\tau{(m-1)}}^{*} H_{\tau{(m)}})^{(\alpha)}\circ\cdots\circ(H_{\tau{(m-3)}}^* H_{\tau{(m-2)}})^{(\alpha)}),$$
we obtain by (\ref{very_good})  that
$$ r((H_{\tau{(1)}}^{*}H_{\tau{(2)}})^{(\alpha)}\circ(H_{\tau{(3)}}^{*}H_{\tau{(4)}})^{(\alpha)}\circ\cdots \circ(H_{\tau{(m-1)}}^{*}H_{\tau{(m)}})^{(\alpha)})\le$$
$$r(S_1^{(\alpha)}\circ S_2^{(\alpha)}\circ\cdots\circ S_{\frac{m}{2}}^{(\alpha)})^{\frac{2}{m}}\le (r(S_1)^{\alpha}\cdots
r(S_{\frac{m}{2}})^{\alpha})^{\frac{2}{m}}$$
$$= r(H_{\tau{(1)}}^{*}H_{\tau{(2)}}H_{\tau{(3)}}^{*}H_{\tau{(4)}}\cdots H_{\tau{(m-1)}}^{*}H_{\tau{(m)}})^{\alpha},$$
where the last equality follows from $r(S_1)=\cdots=r(S_{\frac{m}{2}}).$
\end{proof}

\begin{corollary}
Let $m$ be even, $\alpha\ge\frac{2}{m}$,  $\tau \in \sigma_m$, $\beta \in [0,1]$ and
let $H_1,\ldots,H_m$ be nonnegative matrices that define operators on $l^2(R)$. Let $A_j$ and $B_j$ for
$j=1, \ldots, m $ be as in Corollary \ref{goodAB} and denote $S_i =  A_{i} \cdots A_{\frac{m}{2}} A_{1} \cdots A_{i-1}$ and $T_i =  B_{i} \cdots B_{\frac{m}{2}} B_{1} \cdots B_{i-1}$ for $i=1, \ldots , \frac{m}{2}$.  Then
\be
\label{KA4}
\|H_1^{(\alpha)} \circ\cdots \circ H_m^{(\alpha)}\|\le  r(A_1^{(\alpha)}\circ\cdots \circ A_m^{(\alpha)})^{\frac{\beta}{2}}r(B_1 ^{(\alpha)}\circ\cdots \circ B_m ^{(\alpha)})^{\frac{1-\beta}{2}}
\ee
$$\le r((H_{\tau{(1)}}^{*}H_{\tau{(2)}})^{(\alpha)}\circ( H_{\tau{(3)}}^{*}H_{\tau{(4)}})^{(\alpha)}\circ \cdots \circ(H_{\tau{(m-1)}}^{*}H_{\tau{(m)}})^{(\alpha)})^{\beta}\cdot$$
$$ r((H_{\tau{(1)}}H^{*} _{\tau{(2)}})^{(\alpha)}\circ( H_{\tau{(3)}}H^{*} _{\tau{(4)}})^{(\alpha)}\circ \cdots \circ(H_{\tau{(m-1)}}H^{*} _{\tau{(m)}})^{(\alpha)})^{1-\beta}$$
$$\le r\left(S_1^{(\alpha)}\circ S_2^{(\alpha)}\circ \cdots \circ S_{\frac{m}{2}}^{(\alpha)}\right)^{\frac{2\beta}{m}} r\left(T_1^{(\alpha)}\circ T_2^{(\alpha)}\circ \cdots \circ T_{\frac{m}{2}}^{(\alpha)}\right)^{\frac{2(1-\beta)}{m}}$$
$$
\le r(H_{\tau{(1)}}^{*}H_{\tau{(2)}}H_{\tau{(3)}}^{*}H_{\tau{(4)}}\cdots H_{\tau{(m-1)}}^{*}H_{\tau{(m)}})^{\alpha \beta} r(H_{\tau{(1)}}H^{*} _{\tau{(2)}}H_{\tau{(3)}}H^{*} _{\tau{(4)}}\cdots H_{\tau{(m-1)}}H^{*} _{\tau{(m)}})^{\alpha (1-\beta)}.
$$
\end{corollary}

The following result extends \cite[Theorem 2.13]{Z18} and \cite[Theorem 4.1]{P18a}.

\begin{theorem}
Let $H_1,\ldots,H_m$ be positive kernel operators on $L^{2}(X,\mu)$ and $\{\tau,\nu \} \subset \sigma_m$. Denote
$Q_j=H_{\tau{(j)}}^*H_{\nu{(j)}}\cdots H_{\tau{(m)}}^*H_{\nu{(m)}}H_{\tau{(1)}}^*H_{\nu{(1)}}\cdots H_{\tau{(j-1)}}^*H_{\nu{(j-1)}}$ for $j=1,\ldots,m$. 

\noindent (i) Then
$$\|H_1^{(\frac{1}{m})}\circ\cdots \circ H_m^{(\frac{1}{m})}\|\le r((H_{\tau{(1)}}^*H_{\nu{(1)}})^{(\frac{1}{m})}\circ\cdots \circ(H_{\tau{(m)}}^*H_{\nu{(m)}})^{(\frac{1}{m})})^{\frac{1}{2}}$$
\be
\le r(Q_1^{(\frac{1}{m})}\circ\cdots \circ Q_m^{(\frac{1}{m})})^{\frac{1}{2m}}\le r(H_{\tau{(1)}}^*H_{\nu{(1)}}\cdots H_{\tau{(m)}}^*H_{\nu{(m)}})^{\frac{1}{2m}}.
\label{KAperm}
\ee

\noindent (ii) If $ H_1, \ldots , H_m$ are nonnegative matrices that define operators on $l^2 (R)$ and if $\alpha \ge \frac{1}{m}$, then
$$\|H_1 ^{(\alpha)}\circ\cdots \circ H_m ^{(\alpha)}\|\le r((H_{\tau{(1)}}^*H_{\nu{(1)}})^{(\alpha)} \circ\cdots \circ(H_{\tau{(m)}}^*H_{\nu{(m)}})^{(\alpha)})^{\frac{1}{2}}$$
\be
\le r(Q_1 ^{(\alpha)}\circ\cdots \circ Q_m ^{(\alpha)})^{\frac{1}{2m}}\le r(H_{\tau{(1)}}^*H_{\nu{(1)}}\cdots H_{\tau{(m)}}^*H_{\nu{(m)}})^{\frac{\alpha}{2}}.
\label{KAperm2}
\ee
\label{12}
\end{theorem}
\begin{proof} First we prove (\ref{KAperm}). By (\ref{L2}) and (\ref{very_good}) we have
$$\|H_1^{(\frac{1}{m})}\circ\cdots \circ H_m^{(\frac{1}{m})}\|= r((H_1^{(\frac{1}{m})}\circ\cdots \circ H_m^{(\frac{1}{m})})^*(H_1^{(\frac{1}{m})}\circ\cdots \circ H_m^{(\frac{1}{m})}))^{\frac{1}{2}}=$$
$$r(((H_{\tau{(1)}}^*)^{(\frac{1}{m})}\circ\cdots \circ (H_{\tau{(m)}}^*)^{(\frac{1}{m})})((H_{\nu{(1)}})^{(\frac{1}{m})}\circ\cdots \circ(H_{\nu{(m)}})^{(\frac{1}{m})}))^{\frac{1}{2}}$$
$$\le r((H_{\tau{(1)}}^*H_{\nu{(1)}})^{(\frac{1}{m})}\circ\cdots \circ(H_{\tau{(m)}}^*H_{\nu{(m)}})^{(\frac{1}{m})})^{\frac{1}{2}}.$$
Notice that
$$ ((H_{\tau{(1)}}^*H_{\nu{(1)}})^{(\frac{1}{m})}\circ\cdots \circ(H_{\tau{(m)}}^*H_{\nu{(m)}})^{(\frac{1}{m})})^m=((H_{\tau{(1)}}^*H_{\nu{(1)}})^{(\frac{1}{m})}\circ\cdots \circ(H_{\tau{(m)}}^*H_{\nu{(m)}})^{(\frac{1}{m})})$$
$$((H_{\tau{(2)}}^*H_{\nu{(2)}})^{(\frac{1}{m})}\circ\cdots \circ(H_{\tau{(1)}}^*H_{\nu{(1)}})^{(\frac{1}{m})})\cdots
((H_{\tau{(m)}}^*H_{\nu{(m)}})^{(\frac{1}{m})}\circ\cdots \circ(H_{\tau{(m-1)}}^*H_{\nu{(m-1)}})^{(\frac{1}{m})}).$$
 It follows by (\ref{very_good}) that
$$r((H_{\tau{(1)}}^*H_{\nu{(1)}})^{(\frac{1}{m})}\circ\cdots \circ(H_{\tau{(m)}}^*H_{\nu{(m)}})^{(\frac{1}{m})})^{\frac{1}{2}}\le r(Q_1^{(\frac{1}{m})}\circ\cdots \circ Q_m^{(\frac{1}{m})})^{\frac{1}{2m}}$$
$$\le ( r(Q_1)\cdots r(Q_m))^{\frac{1}{2m^2}}=r(H_{\tau{(1)}}^*H_{\nu{(1)}}\cdots H_{\tau{(m)}}^*H_{\nu{(m)}})^{\frac{1}{2m}},$$
where the last equality follows from $r(Q_1)= \ldots =r(Q_m)$. This
 completes the proof of (\ref{KAperm}). The proof of (\ref{KAperm2}) is similar by applying Theorem \ref{good_work}(ii).
\end{proof}

The following corollary is a refinement of \cite[Inequality (4.11)]{P18a}, which differs from 
refinements in \cite[Inequalities (4.15) and (4.17)]{P18a}. It also extends and generalizes \cite[Corollary 2.15]{Z18}.
\begin{corollary}
\label{lih}
Let $m$ be odd and let $H_1,\ldots,H_m$ be positive kernel operators on $L^{2}(X,\mu)$. 

\noindent (i) Then
$$\|H_1^{(\frac{1}{m})}\circ\cdots \circ H_m^{(\frac{1}{m})}\|$$
$$\le r((H_1^*H_2)^{(\frac{1}{m})}\circ\cdots \circ(H_{m-2}^*H_{m-1})^{(\frac{1}{m})}\circ(H_m^*H_1)^{(\frac{1}{m})}\circ(H_2^*H_3)^{(\frac{1}{m})}\circ\cdots \circ(H_{m-1}^*H_m)^{(\frac{1}{m})})^{\frac{1}{2}}$$
\be
\le r(H_1^*H_2\cdots H_{m-2}^*H_{m-1}H_m^*H_1H_2^*H_3\cdots H_{m-1}^*H_m)^{\frac{1}{2m}}.
\label{lih1}
\ee

\noindent (ii) If $ H_1, \ldots , H_m$ are nonnegative matrices that define operators on $l^2 (R)$ and if $\alpha \ge \frac{1}{m}$, then
$$\|H_1 ^{(\alpha)}\circ\cdots \circ H_m ^{(\alpha)}\|$$
$$\le r((H_1^*H_2)^{(\alpha)}\circ\cdots \circ(H_{m-2}^*H_{m-1})^{(\alpha)}\circ(H_m^*H_1)^{(\alpha)}\circ(H_2^*H_3)^{(\alpha)}\circ\cdots \circ(H_{m-1}^*H_m)^{(\alpha)})^{\frac{1}{2}}$$
\be
\le r(H_1^*H_2\cdots H_{m-2}^*H_{m-1}H_m^*H_1H_2^*H_3\cdots H_{m-1}^*H_m)^{\frac{\alpha}{2}}.
\label{lih2}
\ee

\end{corollary}
\begin{proof}
The result follows by taking  the permutations
$\tau{(j)}=2j-1$ for $1\le j\le\frac{m+1}{2}$; $\tau{(j)}=2(j-\frac{m+1}{2})$ for $\frac{m+3}{2}\le j \le m$ and $\nu{(j)}=2j$ for $1\le j\le\frac{m-1}{2}$; $\nu{(j)}=2(j-\frac{m-1}{2})-1$ for $\frac{m+1}{2}\le j\le m$ in Theorem \ref{12}.
\end{proof}

The following corollary gives new lower bounds for  the operator norm of the Jordan triple product $ABA$,
which differ from the one obtained in \cite[Corollary 4.10]{P18a}. The result follows from Corollary \ref{lih} and Theorem \ref{12} by taking $H_1=A$, $H_2=B^*$ and $H_3=A$. 
\begin{corollary}
 Let $A$ and $B$ be positive kernel operators on $L^2(X, \mu)$. 

 \noindent (i) Then
\begin{eqnarray}
\nonumber
& & \|A^{\left(\frac{1}{3}\right)} \circ (B^*)^{\left(\frac{1}{3}\right)} \circ  A^{\left(\frac{1}{3}\right)} \|  \\
\nonumber
 &\le & r ^{\frac{1}{2}} \left( (A ^* B^*)^{\left(\frac{1}{3}\right)} \circ  
 (A^*A) )^{\left(\frac{1}{3}\right)} \circ (B A)^{\left(\frac{1}{3}\right)} \right) \\
\nonumber
 &\le & r ^{\frac{1}{6}} \left( (A ^* B^*A ^*ABA)^{\left(\frac{1}{3}\right)} \circ 
 ( A^*  ABAA^*B^*)^{\left(\frac{1}{3}\right)} \circ (B AA ^* B^*A^*A) ^{\left(\frac{1}{3}\right)} \right) \\
& \le & \|ABA\|^{\frac{1}{3}}.
\label{tri_lih_jordan}
\end{eqnarray}
\noindent (ii)
If $A$ and $B$  are nonnegative matrices that define operators on $l^2 (R)$ and  if $\alpha \ge \frac{1}{3}$, then
\begin{eqnarray}
\nonumber
& & \|A^{\left(\alpha \right)} \circ (B^*)^{\left( \alpha\right)} \circ A^{\left( \alpha \right)} \| \\
\nonumber
 &\le & r ^{\frac{1}{2}} \left( (A ^* B^*)^{\left(\alpha \right)} \circ  
 (A^*A) )^{\left(\alpha \right)} \circ (B A)^{\left(\alpha \right)} \right) \\
\nonumber
 &\le & r ^{\frac{1}{6}} \left( (A ^* B^*A ^*ABA)^{\left(\alpha \right)} \circ 
 ( A^*  ABAA^*B^*)^{\left(\alpha \right)} \circ (B AA ^* B^*A^*A) ^{\left(\alpha \right)} \right) \\
& \le & \|ABA\|^{\alpha}.
\label{matrix_jordan}
\end{eqnarray}
\label{jordan}
\end{corollary}
The following result generalizes \cite[Inequality (2.12)]{Z18}.
\begin{lemma} Let $\alpha \ge \frac{1}{2}$ and let $C$ be a nonnegative matrix that defines an operator on $l^2 (R)$. Then 
$$r(C^{(\alpha)}\circ (C^*)^{(\alpha)}) \le r(C^{(\alpha)} \circ C^{(\alpha)})\le r(C)^{2\alpha}.$$
\label{Zh}
\end{lemma}
\begin{proof} 
By (\ref{gl1vecr}) in Theorem \ref{good_work}(ii) applied twice we have
$$r(C^{(\alpha)}\circ (C^*)^{(\alpha)}) = r((C^{(\alpha)} \circ C^{(\alpha)})^{(\frac{1}{2})}\circ ((C^*)^{(\alpha)} \circ (C^*)^{(\alpha)})^{(\frac{1}{2})})$$
$$ \le r(C^{(\alpha)} \circ C^{(\alpha)})^{\frac{1}{2}}r ((C^*)^{(\alpha)} \circ (C^*)^{(\alpha)})^{\frac{1}{2}}=r(C^{(\alpha)} \circ C^{(\alpha)})\le r(C)^{2\alpha},$$
which completes the proof.
\end{proof}
The following result generalizes \cite[Theorem 2.17]{Z18} and refines \cite[Inequalities (4.9)]{P18a}. It follows e.g. from Theorem \ref{12} (or \cite[Inequalities (4.9)]{P18a}) and Lemma \ref{Zh}.
\begin{corollary} Let $\alpha \ge \frac{1}{2}$ and let $A$ and $B$ be nonnegative matrices that define operators on $l^2 (R)$. Then 

$$
\|A^{(\alpha)} \circ  B^{(\alpha)} \| \le r ^{\frac{1}{2}} \left ( (A^* B ) ^{(\alpha)}\circ (B ^* A)^{(\alpha)} \right) $$
\be
\le r ^{\frac{1}{2}} \left ( (A^* B ) ^{(\alpha)}\circ (A^* B ) ^{(\alpha)}\right) 
\le r ^{\alpha} (A^* B ),  
\label{matrixT}
\ee
\label{ok2}
\end{corollary}
\begin{remark} {\rm Several results of Section 3 can be further refined by applying  Theorems \ref{ref_powers} and \ref{ugly_ref} in the proofs. We omit the details.
}
\end{remark}

\bigskip

\baselineskip 5mm

{\it Acknowledgments.}
The first author acknowledges a partial support of Erasmus+ European Mobility program (grant KA103), COST Short Term Scientific Mission program (action CA18232) and the Slovenian Research Agency (grants P1-0222 and P1-0288). The first author thanks the colleagues and staff at the Faculty of Mechanical Engineering and Institute of Mathematics, Physics and Mechanics for their hospitality during the research stay in Slovenia. The second author acknowledges a partial support of  the Slovenian Research Agency (grants P1-0222, J1-8133 and J2-2512).\\


\vspace{2mm}

\noindent

\end{document}